\documentclass{amsart}
\usepackage{verbatim}
\newcommand{\cal}{\mathcal}

\newcommand{\rd}{{\mathbb R}^d}

\newcommand{\R}{{\mathbb R}}
\newcommand{\N}{{\mathbb N}}

\newcommand{\ep}{\varepsilon}

\newcommand{\diam}{{\rm diam}\,}

\newcommand{\E}{{\mathbb E}}
\newcommand{\Ha}{{\cal H}}

\newcommand{\ind}[1]{\mathbf{1}_{#1}}
\newcommand{\eps}{r}

\newtheorem{Theorem}{Theorem}
\newtheorem{Proposition}[Theorem]{Proposition}
\newtheorem{Lemma}[Theorem]{Lemma}
\newtheorem{Corollary}[Theorem]{Corollary}
\newtheorem{example}[Theorem]{Example}

\numberwithin{equation}{section} \numberwithin{Theorem}{section}

\begin{document}
\title{On volume and surface area of parallel sets}
\author{Jan Rataj}
\address{Charles University, Faculty of Mathematics and Physics, Sokolovsk\'a 83, 18675 Praha~8, Czech Republic}
\author{Steffen Winter}
\address{Karlsruhe Institute of Technology, Department of Mathematics, 76128 Karlsruhe, Germany}
\thanks{The first author was supported by the Czech Ministry of Education, project no.\ MSM 0021620839. The paper was completed while both authors were supported by a cooperation grant of the Czech and the German science foundation, GACR project no.\ P201/10/J039 and DFG project no.\ WE 1613/2-1.} 
\date{\today}
\begin{abstract}
The $r$-parallel set to a set $A$ in a Euclidean space consists of all points with distance at most $r$ from $A$. We clarify the relation between the volume and the surface area of parallel sets and study the asymptotic behaviour of both quantities as $r$ tends to $0$. We show, for instance, that in general, the existence of a (suitably rescaled) limit of the surface area implies the existence of the corresponding limit for the volume, known as the Minkowski content. A full characterisation is obtained for the case of self-similar fractal sets. Applications to stationary random sets are discussed as well, in particular, to the trajectory of the Brownian motion. 
\end{abstract}
\maketitle

\section{Introduction}
For $r>0$, the $r$-parallel set $A_r$ of a subset $A$ of $\rd$ is the set of all points with distance at most $r$ from $A$. As $r$ tends to $0$, the parallel sets $A_r$ approximate the closure of $A$.  The volume $V_A(r)$ of $A_r$ was investigated by Kneser \cite{Kneser} and later by Stach\'o \cite{Stacho} who also studied the relation to the $(d-1)$-dimensional content of the boundary $\partial A_r$. Recently, Hug et al.\ \cite{HLW04} derived a generalized Steiner formula for closed sets in $\rd$ and obtained as a corollary relations for the volume and surface area of parallel sets strengthening those of Stach\'o. 

Also in fractal geometry, parallel sets play an important role. Minkowski content and Minkowski dimension of $A$ describe the asymptotic behaviour of the volume of $A_r$, as $r\to 0$. The Minkowski dimension (which is equivalent to the box dimension) is an important tool in applications. The Minkowski measurability of a set $A$, i.e.\ the existence of its Minkowski content as a positive and finite number, is deeply connected with the spectral theory on domains "bounded" by $A$, cf.~\cite{lapidus} and the references therein. For self-similar sets in $\R^d$, Gatzouras \cite{gatzouras} gave a characterization of Minkowski measurability and derived formulas for the Minkowski content (in case it exists) and some suitably averaged counterpart. 
The idea of approximation with parallel sets has also been used to introduce certain other geometric quantities for fractal sets:
Winter \cite{winter} and Z\"ahle \cite{Za09} considered (total) curvature measures of the parallel sets (whenever defined in a generalized sense) and introduced appropriately rescaled limits (as $r\to 0$) as fractal curvatures.

Parallel sets have also been used to approximate the highly irregular trajectory of the Brownian motion. Formulas for the mean volume of the parallel sets are known for decades. Recently, also the mean surface area has been investigated (\cite{RSS07}), as well as other curvature functionals (\cite{Last06}, \cite{RSM09}). 

In this note we investigate more deeply the connection between the volume $V_A(r)$ and the surface area $S_A(r):=\Ha^{d-1}(\partial A_r)$ of parallel sets to a set $A\subset\rd$ ($\Ha^{d-1}$ denotes the $(d-1)$-dimensional Hausdorff measure). In Section~2, we strengthen slightly a result from \cite{HLW04}, using a rectifiability argument, and obtain that $V_A'(r)=S_A(r)$, up to countably many $r$'s. In Section~3 we use this result to compare the asymptotic behaviour of surface area and volume. We introduce here, in analogy to the Minkowski content, the \emph{surface area based content} and \emph{surface area based dimension} which under additional assumptions coincide with the Minkowski's quantities. To illustrate this relation, consider the case when $A$ is a $(d-1)$-dimensional $C^2$ smooth compact submanifold of $\rd$. Then, both $V_A(r)/2r$ and $S_A(r)/2$ converge to the same limit as $r\to 0$, namely to the $(d-1)$-dimensional Minkowski content of $A$ which equals $\Ha^{d-1}(A)$ in this case. We show that some analogous results hold for general compact sets with zero volume (and arbitrary dimension). This is closely related to a conjecture that has been communicated by Martina Z\"ahle to the authors: If for a self-similar set $A$ of dimension $D$ the total curvatures $C_k(A_r)$ of the parallel sets are defined, then the (appropriately rescaled and averaged) limits of these quantities coincide for all integers $k>D-1$. 
In Section~4, we study the class of self-similar sets more closely and show that the aforementioned surface area based content (the limit of the total curvatures of order $d-1$) coincides with the Minkowski content
provided the set is non-arithmetic while the corresponding averaged limits coincide in general; this  partially confirms the above conjecture. It also extends and sheds a new light on some results in \cite{winter}.
Finally, Section~5 deals with mean values for stationary random closed sets. As particular applications, we strengthen the results on mean surface area of the parallel set to the Brownian motion from \cite{Last06} and \cite{RSS07}, and derive some estimates on the asymptotics of the surface area.

\section{Surface area content of the parallel sets}
Let $A$ be a bounded subset of $\rd$ and $r>0$. Denote by $d_A$ be the (Euclidean) distance function of the set $A$, and by $A_r$ and $A_{<r}$  the closed and open $r$-parallel neighbourhood of $A$, respectively, i.e.,
$$A_r=\{ z\in\rd:\, d_A(z)\leq r\},\quad
A_{<r}=\{ z\in\rd:\, d_A(z)<r\}.$$
Finally, let $V_A(r)={\cal H}^d(A_r)$ be the volume of the $r$-parallel neighbourhood.

Stach\'o \cite{Stacho} showed (using the results of Kneser \cite{Kneser}) that the left and right derivatives $(V_A)'_+(r)\leq (V_A)'_-(r)$ exist at any $r>0$ and are equal up to countably many $r$'s. He also showed that for any $r>0$,
\begin{equation} \label{Stach}
{\cal M}^{d-1}(\partial A_{<r})=\frac{(V_A)'_-(r)+(V_A)'_+(r)}2,
\end{equation}
where ${\cal M}^{d-1}$ is the $(d-1)$-dimensional Minkowski content. Recently, Hug, Last and Weil \cite{HLW04} proved a generalized Steiner formula for closed sets and they obtained as corollary the relation (see \cite[Corollary~4.6]{HLW04})
$$(V_A)'_+(r)={\cal H}^{d-1}(\partial_+A_r),\quad r>0,$$
where ${\cal H}^{d-1}$ is the $(d-1)$-dimensional Hausdorff measure and $\partial_+Z$ is the set of all boundary points $z\in Z$ for which there exists a point $y\not\in Z$ with $d_Z(y)=|y-z|$. We thus have the inequalities
\begin{eqnarray*}
&(V_A)'_+(r)&\leq {\cal M}^{d-1}(\partial A_{<r})\leq (V_A)'_-(r)\\
&\parallel&\\
&{\cal H}^{d-1}(\partial_+A_r)&\leq{\cal H}^{d-1}(\partial A_r)\leq {\cal H}^{d-1}(\partial A_{<r}).
\end{eqnarray*}
The following two examples confirm that none of the inequalities above can be replaced by equality.

\begin{example} \rm
Let $A$ be the union of two unit parallel line segments of distance $2r$ in $\R^2$. Then ${\cal H}^{1}(\partial A_r)=2+4\pi r<3+4\pi r={\cal H}^{1}(\partial A_{<r})$. Also, $(V_A)'_+(r)=2+4\pi r<3+4\pi r=(V_A)'_-(r)$, which, together with \eqref{Stach}, implies that the inequalities on the first line of the diagram are sharp.
\end{example}

\begin{example} \rm
Let $C$ be a totally disconnected compact subset with positive one-dimensional measure of the segment $[(0,0),(1,0)]$ in $\R^2$ and let $A=C\cup (C+2re_2)$, where $e_2=(0,1)$. Then $C+re_2$ belongs to the boundary $\partial A_r$, but not to $\partial_+ A_r$, hence, ${\cal H}^1(\partial_+ A_r)<{\cal H}^1(\partial A_r)$.
\end{example}

Hug et al.\ claim in \cite[p.~257]{HLW04} that $\partial A_{<r}$ need not be $(d-1)$-rectifiable and, hence, the equality of the Minkowski content and Hausdorff measure does not follow in general. However, the counterexamples of Ferry \cite{Ferry} mentioned in \cite[p.~257]{HLW04} only show that $\partial A_{<r}$ need not be a $(d-1)$-manifold and do not disclaim rectifiability. We shall show that, in fact, $\partial A_{<r}$ is $(d-1)$-rectifiable for all $r>0$ and, as a consequence, derive the equivalence of Minkowski content and Hausdorff measure of $\partial A_{<r}$ in general.
We recall that a set is $k$-rectifiable if it is a Lipschitz image of a bounded subset of $\R^k$.

\begin{Proposition} \label{T1}
If $A\subseteq\rd$ is bounded then $\partial A_{<r}$ and $\partial A_r$ are $(d-1)$-rectifiable for any $r>0$.
\end{Proposition}

\begin{proof}
Since $\partial A_r\subseteq \partial A_{<r}$, it suffices to prove the rectifiability of $\partial A_{<r}$, and,
since $A_{<r}=\overline{A}_{<r}$ for any $r>0$, we can assume without loss of generality that $A$ is compact.
Given a point $z\not\in A$, denote
$$\Sigma_A(z)=\{ a\in A:\, |z-a|=d_A(z)\},$$
the set of all nearest points of $A$ to $z$. The point $z$ is called {\it regular} if $z$ does not belong to the convex hull of $\Sigma_A$. A number $r>0$ is called a {\it regular value of} $d_A$ if all points of $\partial A_{<r}$ are regular (cf.\ \cite{Ferry}). Fu \cite{Fu85} showed that $\partial A_{<r}$ is a Lipschitz manifold if $r$ is a regular value. It is not difficult to see that if $r>\diam A$ then $r$ is a regular value of $d_A$.\footnote{One can even show that there is no critical value of $d_A$ greater than $\sqrt{d/(2d+2)}\,\diam A$ and this upper bound is attained if $A$ consists of the vertices of a regular $d$-simplex, cf.\ \cite[p.~1038]{Fu85}.} For $r<\diam A$, partition $A$ into finitely many subsets of diameters less than $r$, $A=E^1\cup\cdots\cup E^k$, and note that $\partial A_{<r}\subseteq\partial (E^1)_{<r}\cup\cdots\cup \partial (E^k)_{<r}$. Each of the sets $\partial (E^k)_{<r}$ is a Lipschitz manifold and, since it is compact, it is $(d-1)$-rectifiable. Since rectifiability is preserved by finite unions, it follows that $\partial A_{<r}$ is $(d-1)$-rectifiable.
\end{proof}

Applying now \cite[\S3.2.39]{F69}, we get
\begin{Corollary} \label{Cor1}
For any $r>0$, we have 
$$
{\cal M}^{d-1}(\partial A_{r})={\cal H}^{d-1}(\partial A_{r}) \text{ and } {\cal M}^{d-1}(\partial A_{<r})={\cal H}^{d-1}(\partial A_{<r}).
$$
\end{Corollary}

With the result of Stach\'o we get immediately the following strengthening of \cite[Corollary~4.7]{HLW04} and \cite[Lemma~1]{HHL04}:
\begin{Corollary} \label{Cor2}
The function $V_A$ is differentiable at $r>0$ with 
$$V_A'(r)={\cal H}^{d-1}(\partial A_{r})={\cal H}^{d-1}(\partial_+ A_{r})={\cal H}^{d-1}(\partial A_{<r})$$
for all $r>0$ up to a countable set.
\end{Corollary}

\begin{Corollary} \label{Cor3}
For any $r>0$,
$${\cal H}^{d-1}(\partial A_{r})\leq\lim_{s\to r_-}V_A'(s)$$
(the limit is understood over those $s<r$ where $V_A'$ exists).
\end{Corollary}

\begin{proof}
Using Corollary~\ref{Cor1} and \eqref{Stach}, we get
$${\cal H}^{d-1}(\partial A_{r})={\cal M}^{d-1}(\partial A_{r})\leq {\cal M}^{d-1}(\partial A_{<r})\leq (V_A)'_-(r).$$
The assertion follows now from the left continuity of $(V_A)'_-$ and from the fact that the left derivative of $V_A$ coincides with the derivative up to a countable set of values, see \cite{Stacho}.
\end{proof}

\section{Asymptotic behaviour}

Given a compact set $A\subseteq\R^d$, we shall use the notation $S_A(r):=\Ha^{d-1}(\partial A_r)$, $r\geq 0$, for the $(d-1)$-dimensional Hausdorff measure of the boundary of the parallel set. We shall discuss in this section the asymptotic behaviour of $S_A(r)$ as $r\to 0$. This is, of course, closely related to the asymptotic behaviour of $V_A(r)$ through Corollary~\ref{Cor2}.

Recall the \emph{$s$-dimensional lower and upper Minkowski content} of a compact set $A\subset\R^d$, which are defined by
\[
\underline{\cal M}^s(A):=\liminf_{r\to 0} \frac{V_A(r)}{\kappa_{d-s}r^{d-s}} \quad \text{ and } \quad 
\overline{\cal M}^s(A):=\limsup_{r\to 0} \frac{V_A(r)}{\kappa_{d-s}r^{d-s}},
\] 
where $\kappa_t:=\pi^{t/2}/\Gamma(1+\frac t2)$. (If $t$ is an integer, then $\kappa_t$ is the volume of a unit $t$-ball).
If $\underline{\cal M}^s(A)=\overline{\cal M}^s(A)$, then the common value ${\cal M}^s(A)$ is the \emph{$s$-dimensional Minkowski content} of $A$. 
We denote by 
\[
\underline{\dim}_M A:=\inf\{t\ge 0 : \underline{\cal M}^s(A)=0\}=\sup\{t\ge 0 :\underline{\cal M}^s(A)=\infty\}
\]
and 
\[
\overline{\dim}_M A=\inf\{t\ge 0 :\overline{\cal M}^s(A)=0\}=\sup\{t\ge 0 :\overline{\cal M}^s(A)=\infty\}
\]
the \emph{lower} and \emph{upper Minkowski dimension} of $A$.

In analogy with the Minkowski content, we define 
for $0\leq s<d$
\[
\underline{\cal S}^s(A):=\liminf_{r\to 0} \frac{S_A(r)}{(d-s)\kappa_{d-s}r^{d-1-s}} 
\quad \text{ and } \quad 
\overline{\cal S}^s(A):=\limsup_{r\to 0} \frac{S_A(r)}{(d-s)\kappa_{d-s}r^{d-1-s}}.
\]  
If both numbers coincide, we 
denote the common value by ${\cal S}^s(A)$ and call it the \emph{surface area based content} or, briefly, the \emph{S-content} of $A$. For convenience, we set ${\cal S}^d(A):=0$. (Note that the above definition would not make sense for $s=d$. However, setting ${\cal S}^d(A)$ zero is justified by the fact that $\lim_{r\to 0}rS_A(r)=0$. Indeed, by Corollary~\ref{Cor1} and \cite[Satz~4]{Kneser}, we have $\Ha^{d-1}(\partial A_r)={\cal M}^{d-1}(\partial  A_r)\le \frac dr (V_A(r)-V_A(0))$ for each $r>0$. Since $(V_A(r)-V_A(0))\to 0$ as $r\to 0$,  we obtain\\ $\lim_{r\to 0} r \Ha^{d-1}(\partial A_r) = 0$, as claimed.) 

We call the numbers 
\[
\underline{\dim}_S A:=\sup\{0\le t\le d: \underline{\cal S}^t(A)=\infty\}=\inf\{0\le t\le d: \underline{\cal S}^t(A)=0\}
\]
and
\[
\overline{\dim}_S A:=\sup\{0\le t\le d: \overline{\cal S}^t(A)=\infty\}=\inf\{0\le t\le d: \overline{\cal S}^t(A)=0\}
\]
the \emph{lower and upper surface area based dimension} or \emph{S-dimension} of $A$, respectively.  Obviously, $\underline{\dim}_S A\le \overline{\dim}_S A$, and if equality holds, the common value will be regarded as the \emph{surface area based dimension} (or \emph{S-dimension}) of $A$ and denoted by $\dim_S A$. 

\noindent{\bf Remark.}
The upper S-dimension $\overline{\dim}_S A$ is closely related to the curvature scaling exponent $s_{d-1}(A)$ defined in \cite{winter}. In fact, it is the natural extension of this concept to arbitrary compact sets (just with a different normalization).  Since $2 C_{d-1}(A_r)= \Ha^{d-1}(\partial A_r)$,  whenever $C_{d-1}(A_r)$ is defined, one has $\overline{\dim}_S A= s_{d-1}-1+d$.  Similarly, ${\cal S}^{s}(A)$, with $s=\dim_S A$, generalizes $C_{d-1}^f(A)$, the \emph{fractal curvature} of order $d-1$. If the latter exists, then these numbers differ by the constant factor $\frac 12 \kappa_{d-s} (d-s)$. 
\vspace{2mm}

\begin{Proposition} \label{prop3}
Let $A\subseteq\rd$ be compact and let $h:[0,\infty)\to[0,\infty)$ be a continuous differentiable function with $h(0)=0$. 
Assume that $h'$ is nonzero on some right neighbourhood of $0$.
Let
\[
 \underline{S}:=\liminf_{r\to 0} \frac{S_A(r)}{h'(r)} \qquad \text{ and }\qquad  \overline{S}:=\limsup_{r\to 0} \frac{S_A(r)}{h'(r)}.
\]
Then 
\begin{equation} \label{ineq1}
\underline{S}\le \liminf_{r\to 0}\frac{V_A(r)-V_A(0)}{h(r)}\le \limsup_{r\to 0}\frac{V_A(r)-V_A(0)}{h(r)}\le \overline{S}.
\end{equation}
In particular, if $\underline{S}=\overline{S}$, i.e.~if the limit $$S:=\lim_{r\to 0}\frac{S_A(r)}{h'(r)}\in[0,\infty]$$
exists then 
$$\lim_{r\to 0}\frac{V_A(r)-V_A(0)}{h(r)}$$
exists as well and equals $S$.
\end{Proposition}

\begin{proof}
We follow the lines of the classical proof of l'Hospital's rule, using the absolute continuity of $V_A$ (see e.g.\ \cite[Lemma~2]{Stacho}).

We shall use the following fact:
{\sl For any $r>0$ there exist $0<t_1,t_2<r$ such that} 
\begin{equation} \label{E22}
\frac{(V_A)'(t_1)}{h'(t_1)}\leq\frac{V_A(r)-V_A(0)}{h(r)}\leq \frac{(V_A)'(t_2)}{h'(t_2)}.
\end{equation}
To see this, fix an $r>0$. Since the function
$$\Phi(t):=(V_A(r)-V_A(0))h(t)-h(r)V_A(t),\quad 0\leq t\leq r,$$
is absolutely continuous and $\Phi(0)=\Phi(r)$,
we have $\int_0^r\Phi'(t)\, dt=0$. Hence, either $\Phi'(t)=0$ for almost all $t\in (0,r)$, or there exist $t_1,t_2\in (0,r)$ such that $\Phi'(t_1)>0>\Phi'(t_2)$. This proves \eqref{E22}.

Note that $V_A'(t)=S_A(t)$ whenever $V_A'(t)$ exists. Thus, taking the $\limsup$ as $r\to 0$ in the right inequality in \eqref{E22}, we get
\[
\limsup_{r\to 0} \frac{V_A(r)-V_A(0)}{h(r)}\le \limsup_{r\to 0} \frac{S_A(t_2(r))}{h'(t_2(r))}\le \limsup_{r\to 0} \frac{S_A(r)}{h'(r)}=\overline{S},
\] 
which is the right inequality in \eqref{ineq1}. Analogously, the left inequality is obtained by taking the $\liminf$ as $r\to 0$ in left inequality in \eqref{E22}. 

If $\underline{S}=\overline{S}$, then the existence of the limit $\lim_{r\to 0}\frac{V_A(r)-V_A(0)}{h(r)}$ follows immediately from \eqref{ineq1}. 
\end{proof}
Proposition~\ref{prop3} yields the following general relations between Minkowski content and S-content.
\begin{Corollary} \label{cor4}
Let $A\subset\R^d$ be compact and assume that $V_A(0)=0$. Then, for all $s\leq d$, 
$$\underline{\cal S}^s(A)\le \underline{\cal M}^s(A)\le \overline{\cal M}^s(A)\le \overline{\cal S}^s(A).$$
\end{Corollary}

\begin{proof}
In the case $s=d$, we have $\underline{\cal S}^d(A)=\overline{\cal S}^d(A)=0$ by definition, and it follows from the assumption and the continuity of the volume function that $\underline{\cal M}^d(A)= \overline{\cal M}^d(A)=0$ as well.

Fix now $s<d$ and let $h(t)=\kappa_{d-s}t^{d-s}$. Applying Proposition~\ref{prop3}, we get
\[ \underline{\cal S}^s(A)=\liminf_{r\to 0} \frac{S_A(r)}{h'(r)}\le \liminf_{r\to 0}\frac{V_A(r)}{h(r)}=\underline{\cal M}^s(A).
\] 
The relation $\overline{\cal M}^s(A)\le \overline{\cal S}^s(A)$ is obtained analogously by applying the third inequality from \eqref{ineq1}.
\end{proof}

It is obvious, that the middle inequality in Corollary~\ref{cor4} can be strict. There are many sets for which the Minkowski content does not exist.
However, it was not immediately clear, whether the left and right inequalities can be strict or whether, in fact, equality holds in general.
The following example illustrates that all three inequalities in Corollary~\ref{cor4} can be strict. 

\begin{example}\rm
The Sierpinski gasket $F$ is the self-similar set in $\R^2$ generated by the three similarities $\Phi_1(x)=\frac{1}{2} x$,  $\Phi_2(x)=\frac{1}{2} x+ (\frac 12,0)$ and $\Phi_3(x)=\frac{1}{2} x+(\frac 14,\frac{\sqrt{3}}{4})$. It is well known that its Minkowski dimension is $D:=\dim_M F =\frac{\ln 3}{\ln 2}$. We compute its upper and lower ($D$-dim.) Minkowski contents and S-contents directly. It turns out that all four values are different, providing an example where all inequalities in Corollary~\ref{cor4} are strict.  

Observe that the diameter of $F$ is one and that the inradius $u$ of the middle cut out triangle is $u:=\frac 1{4\sqrt{3}}$. It is not difficult to see that for $n\in\N$ and $r\in I_n:=[2^{-n}u, 2^{n-1}u)$, the area and boundary length of $F_r$ are given by
\[
V(r)=\left(\pi - \frac 32\sqrt{3}(3^n-1)\right) r^2 +3\left(\frac32\right)^n r +\sqrt{3}\left(\frac32\right)^n 2^{-n-2}
\]
and
\[
S(r)=\left(2\pi - 3\sqrt{3} (3^n-1) \right) r + 3\left(\frac32\right)^n. 
\]
Let $t_n(\alpha):=\alpha 2^{-n}u$, $\alpha\in[1,2)$, be a parametrization of $I_n$. We have $t_n(\alpha)^{1-D}=\alpha^{1-D}u^{1-D}\left(\frac32\right)^n$ and thus
\[
\frac{S(t_n(\alpha))}{t_n(\alpha)^{1-D}}=\alpha^D u^D\left(\frac{2\pi+3\sqrt{3}}{3^n}-3\sqrt{3}\right)+\alpha^{D-1} 3u^{D-1}=\alpha^D c_n + \alpha^{D-1} b,
\]  
where $b:=3u^{D-1}$ and $c_n:=u^D \left(3^{-n}(\pi +3\sqrt{3})-3 \sqrt{3}\right)$. Clearly, if we choose the sequence $(\alpha_n)$ in $[1,2]$ such that $\alpha_n$ is the value where the function $\alpha^D c_n + \alpha^{D-1} b$ attains its maximum in $[1,2]$, then
\[
\lim_{n\to\infty} \alpha_n^D c_n + \alpha_n^{D-1} b=\lim_{n\to\infty} \frac{S(t_n(\alpha_n))}{t_n(\alpha_n)^{1-D}}=(2-D)\kappa_{2-D}\overline{\cal S}^D(F).
\]
Moreover, since $c_n\to c:=-3\sqrt{3}u^D$ as $n\to\infty$ and since the function $g:\R^2\to\R, (x,y)\mapsto x^D y+x^{D-1} b$ is continuous, we have $\lim_{n\to\infty} g(\alpha_n,c_n)=g(\alpha_{\max},c)$, where $\alpha_{\max}=\lim_{n\to\infty}\alpha_n$ is the (unique) value where the maximum of the function $g(x,c)$ in $[1,2]$ is attained. A simple calculation shows that $\alpha_{\max}=4(1-\frac 1D)$. Hence,  
\begin{eqnarray*}
\kappa_{2-D}\overline{\cal S}^D(F)&=&\frac{g(\alpha_{\max},c)}{2-D} =\frac{\alpha_{\max}^D c +\alpha_{\max}^{D-1} b}{2-D}\\
&=& \frac{3\sqrt{3}^{1-D} }{(2-D)(D-1)}\left(1-\frac 1D\right)^D\approx 1.846.
\end{eqnarray*}
Choosing the sequence $(\alpha_n)$ such that the minima are attained, a similar argument shows that $\alpha_{\min}=1$ and hence
\begin{eqnarray*}
\kappa_{2-D}\underline{\cal S}^D(F)&=&\frac{g(\alpha_{\min},c)}{2-D} =\frac{c + b}{2-D}=\frac{\sqrt{3}^{1-D}}{2-D}\approx 1.747.
\end{eqnarray*} 
For upper and lower Minkowski content we can argue in the same way. We have
\[
\frac{V(t_n(\alpha))}{t_n(\alpha)^{2-D}}=\alpha^{D} \frac 12 c_n +\alpha^{D-1} b+\alpha^{D-2} b, 
\]
with $b$ and $c_n$ as above.  Now we consider the function $h:\R^2\to\R, (x,y)\mapsto x^D y\frac 12 +x^{D-1} b+x^{D-2}b$. Choosing $\alpha_n$ such that the maximum of $h(x,c_n)$ in $[1,2]$ is attained, we have
\[
\kappa_{2-D}\overline{\cal M}^D(F)=\lim_{n\to\infty} \frac{V(t_n(\alpha_n))}{t_n(\alpha_n)^{2-D}}=\lim_{n\to\infty} h(\alpha_n,c_n)=h(\alpha_{\max},c),
\]
where $c=\lim_n c_n=- 3 \sqrt{3}u^D$ and $\alpha_{\max}=\lim_n \alpha_n$ 
is the value, where $h(x,c)$ attains its maximum in $[1,2]$ (similarly for the minima and the lower Minkowski content).
It turns out that $\alpha_{\max/\min}= \frac{4}{D}\left(D-1\pm\sqrt{\frac 32 D^2-3D+1}\right)
$
and thus 
we obtain
\begin{eqnarray*}
\kappa_{2-D}\overline{\cal M}^D(F)&=& h(\alpha_{\max},c) =\ldots\approx 1.814\\
\kappa_{2-D}\underline{\cal M}^D(F)&=& h(\alpha_{\min},c) =\ldots \approx 1.811.
\end{eqnarray*} 
This shows that for the Sierpinski gasket $F$ we have the strict inequalities
$$ 
\underline{\cal S}^D(F)<\underline{\cal M}^D(F)<\overline{\cal M}^D(F)<\overline{\cal S}^D(F).
$$
\end{example}
The relations between Minkowski content and S-content in Corollary~\ref{cor4} are obtained only in the case when $V_A(0)=0$. Indeed, it is easy to see 
that e.g.\ for the unit square we have Minkowski dimension equal to 2, 
while the S-dimension equals 1. This could be repaired by replacing the volume $V_A(r)$ with the difference $V_A(r)-V_A(0)$ in the definitions of the Minkowski content and dimension; we shall, however, not follow this way here. Note that in both our applications, self-similar fractal sets and Brownian path, the sets considered have zero volume. 

For the dimensions, it is clear from the definitions that, in general, one has $\underline{\dim}_S A\le \overline{\dim}_S A\le d$ and $\underline{\dim}_M A\le\overline{\dim}_M A\le d$. From Corollary~\ref{cor4} we get

\begin{Corollary} \label{cor5}
For $A\subset\R^d$ compact, we have
\begin{enumerate}
\item[(i)] $\underline{\dim}_S A\le\underline{\dim}_M A$\,, 
\item[(ii)] $\overline{\dim}_M A\le\overline{\dim}_S A$, provided $V_A(0)=0$\,.
\end{enumerate}
\end{Corollary}
\begin{proof}
If $\underline{\dim}_MA=d$, assertion (i) is obvious. 
So assume $\underline{\dim}_MA<d$ (which implies $V_A(0)=0$). For each $\underline{\dim}_M A<s\le d$, we have, by Corollary~\ref{cor4}, $\underline{\cal S}^s(A)\le\underline{\cal M}^s(A)=0$ and hence  $\underline{\dim}_S A=\inf\{t: \underline{\cal S}^t(A)=0\}\le s$. Since this holds for $s$ arbitrary close to $\underline{\dim}_M A$, we get $\underline{\dim}_S A\le\underline{\dim}_M A$ as claimed.\\ 
If $V_A(0)=0$, assertion (ii) follows by a similar argument as (i). 
\end{proof}
In the following, we shall show that even $\overline{\dim}_M A=\overline{\dim}_S A$ holds whenever $V_A(0)=0$.

\begin{Lemma} \label{L-34}
If $0\leq s<d$ then
$$\limsup_{r\to 0}\frac{V_A(r)}{r^{d-s}}\geq \frac{d-s}d \limsup_{r\to 0}\frac{S_A(r)}{(d-s)r^{d-s-1}}.$$
\end{Lemma}

\begin{proof}
Using Corollary~\ref{Cor3}, we find that there exists
a sequence $(r_i)$ of differentiability points of $V_A$ decreasing monotonely to $0$ and such that  $$\lim_{i\to\infty}\frac{V'_A(r_i)}{(d-s)r_i^{d-s-1}}=
\limsup_{r\to 0}\frac{S_A(r)}{(d-s)r^{d-s-1}}=:a\in [0,\infty].$$
For all $r_{i+1}\leq r\leq r_i$ such that $V_A'(r)$ exists (which is the case for $\Ha^1$-a.a.\ $r$), we have
$$\frac{V'_A(r_i)}{r_i^{d-1}}\leq \frac{V'_A(r)}{r^{d-1}}\leq \frac{V'_A(r_{i+1})}{r_{i+1}^{d-1}},$$
see \cite[Theorem~1]{Stacho}. Hence,
\begin{eqnarray*}
V_A(r_i)&=&\int_0^{r_i}V'_A(r)\, dr =\sum_{j=i}^\infty \int_{r_{j+1}}^{r_j} V'_A(r)\, dr\\
&\geq& \sum_{j=i}^\infty \int_{r_{j+1}}^{r_j} V'_A(r_j) \frac{r^{d-1}}{r_j^{d-1}}\, dr\\
&=& \sum_{j=i}^\infty V'_A(r_j) \frac{r_j^d-r_{j+1}^d}{dr_j^{d-1}}\\
&=& \sum_{j=i}^\infty \frac{V'_A(r_j)}{(d-s)r_j^{d-s-1}} \frac{d-s}d \frac{r_j^d-r_{j+1}^d}{r_j^s}\\
&\geq& \sum_{j=i}^\infty \frac{V'_A(r_j)}{(d-s)r_j^{d-s-1}} \frac{d-s}d (r_j^{d-s}-r_{j+1}^{d-s}).
\end{eqnarray*}
If $a'<a$ then $\frac{V'_A(r_j)}{(d-s)r_j^{d-s-1}}\geq a'$ for all sufficiently large $j$. Thus, for $i$ large enough, we have
$$\frac{V_A(r_i)}{r_i^{d-s}}\geq \frac{a'\frac{d-s}d \sum_{j=i}^\infty(r_j^{d-s}-r_{j+1}^{d-s})}
{\sum_{j=i}^\infty (r_j^{d-s}-r_{j+1}^{d-s})} =a'\frac{d-s}d,$$
which completes the proof.
\end{proof}

\begin{Corollary} \label{cor:equpper}
Let  $A\subset\rd$ be a compact set. Then, for any $0\leq s\leq d$, 
\[
\overline{\cal M}^s(A)\geq \frac{d-s}d \overline{\cal S}^s(A).
\] 
Consequently, $\overline{\dim}_M A=\overline{\dim}_S A$ whenever $V_A(0)=0$.
\end{Corollary}

Curiously, the analogous method fails when trying to show that $\underline{\dim}_M A=\underline{\dim}_S A$. 
A weaker reversed inequality can be derived from the isoperimetric inequality. 

\begin{Proposition} \label{isoperi} Let $A\subset\rd$ be a compact set. Then, for $0\le s\le d$,
$$
\underline{\cal S}^{s\frac{d-1}{d}}(A)\ge c \left(\underline{\cal M}^{s}(A)\right)^{\frac {d-1}d}\,, 
$$ 
where $c$ is a constant depending only on $s$ and $d$.
Consequently, $\underline{\dim}_S A\ge \frac {d-1}d \underline{\dim}_M A$. 
\end{Proposition}

\begin{proof} 
By the isoperimetric inequality (cf.~Federer~\cite[§3.2.43]{F69}) and Corollary~\ref{Cor1}, we have for each $r>0$
\[
 d \kappa_d^{1/d} V_A(r)^{(d-1)/d} \le \underline{\cal M}^{d-1}(\partial A_r)=\Ha^{d-1}(\partial A_r)=S_A(r).
\]
Fix some $s\le d$ and set $s':=\frac{d-1}{d} s$. Dividing by $(\kappa_{d-s} r^{d-s})^{d-1/d}=\kappa_{d-s}^{(d-1)/d} r^{d-1-s'}$, we get for each $r>0$
\[
\left(\frac{V_A(r)}{\kappa_{d-s} r^{d-s}}\right)^{(d-1)/d}
\le \frac{1}{d \kappa_d^{1/d} \kappa_{d-s}^{(d-1)/d}} \frac{S_A(r)}{r^{d-1-s'}}
= c\,\frac{S_A(r)}{(d-s')\kappa_{d-s'} r^{d-1-s'}},
\]
with
\[ 
c:=\frac{(d-s')\kappa_{d-s'}}{d \kappa_d^{1/d}\kappa_{d-s}^{(d-1)/d}}. 
\]
We can assume $\underline{\cal S}^{s'}(A)<\infty$, since the statement is trivial for $\underline{\cal S}^{s'}(A)=\infty$. Choose a null sequence $(r_n)_{n\in\N}$ such that the limes inferior $\underline{\cal S}^{s'}(A)$ is attained, i.e.~such that
\[
\lim_{n\to\infty} \frac{S_A(r_n)}{(d-s')\kappa_{d-s'} r_n^{d-1-s'}} = \underline{\cal S}^{s'}(A) \in[0,\infty).
\]
Then for each $a>\underline{\cal S}^{s'}(A)$ and $n$ sufficiently large, we have 
\[ 
\frac{S_A(r_n)}{(d-s')\kappa_{d-s'} r_n^{d-1-s'}} \le a
\text{ and thus}
\frac{V_A(r_n)}{\kappa_{d-s} r_n^{d-s}}\le c\, a^{d/(d-1)}.
\]
Letting $n\to \infty$, we obtain 
\[
\left(\underline{\cal M}^s(A)\right)^{(d-1)/d}\le \left(\liminf_{n\to\infty}\frac{V_A(r_n)}{\kappa_{d-s} r_n^{d-s}}\right)^{(d-1)/d}\le c\, a\,,
\]
and since this holds for all $a>\underline{\cal S}^{s'}(A)$, the first inequality follows.\\
The second inequality is an immediate consequence.
\end{proof}

Corollary~\ref{cor4} shows in particular that, for sets of zero volume, the existence of the S-content enforces the existence of the Minkowski content and it also determines its value:
If $\underline{\cal S}^s(A)=\overline{\cal S}^s(A)$ for some $s\leq d$, then also $\underline{\cal M}^s(A)= \overline{\cal M}^s(A)$ and the common value is ${\cal M}^s(A)={\cal S}^s(A)$. In particular, if $0<{\cal S}^s(A)<\infty$ for some $s<d$, then the set $A\subset\R^d$ is \emph{Minkowski measurable}, i.e.\ $0<{\cal M}^s(A)<\infty$. Note that our results do not allow the converse conclusion. The existence of Minkowski content does not seem to imply the existence of the S-content.
\vskip 2mm

\noindent {\bf Remark.}
In fact, and as pointed out by the referee, Proposition~\ref{prop3} and Lemma~\ref{L-34} (and also Lemmas~\ref{L2} and \ref{lem:av-content} in the next section) remain true in the slightly more general and purely analytic setting of Kneser functions. A real continuous function $f$ on $[0,\infty)$ is a {\it Kneser function} if 
$$
f(\lambda b)-f(\lambda a)\leq\lambda^n(f(b)-f(a)),\quad 0<a\leq b,\quad \lambda\geq 1.
$$
The volume function $V_A$ of any bounded set $A\subset\rd$ is a Kneser function, cf.\ \cite{Kneser,Stacho}. All the four above mentioned results can be formulated for Kneser functions instead of $V_A$, since all the properties of $V_A$ used in the proofs are consequences of the Kneser property. As an illustration, we reformulate here Proposition~\ref{prop3}:
\vskip 2mm

{\it Let $f$ be a Kneser function and let $h:[0,\infty)\to[0,\infty)$ be a continuous differentiable function with $h(0)=0$. 
Assume that $h'$ is nonzero on some right neighbourhood of $0$.
Let $\underline{S}:=\liminf_{r\to 0} f'(r)/h'(r)$ and $\overline{S}:=\limsup_{r\to 0} f'(r)/h'(r)$. Then 
$$
\underline{S}\le \liminf_{r\to 0}\frac{f(r)-f(0)}{h(r)}\le \limsup_{r\to 0}\frac{f(r)-f(0)}{h(r)}\le \overline{S}.
$$
In particular, if $\underline{S}=\overline{S}$, i.e.~if the limit $S:=\lim_{r\to 0}f'(r)/h'(r)\in[0,\infty]$ exists then 
$\lim_{r\to 0}(f(r)-f(0))/h(r)$ exists as well and equals $S$.}
\vskip 2mm

\noindent {\bf Remark.} 
Since upper Minkowski and upper S-dimension always coincide, cf.~Corollary~\ref{cor:equpper}, it is a natural question to ask whether the same is true for the lower counterparts, i.e., whether the result obtained in Proposition \ref{isoperi} can be improved. After this paper was submitted, further investigations of the second author revealed that there exist sets whose lower S-dimension is strictly smaller than their lower Minkowski dimension. Even more, the estimate regarding the lower dimensions in Proposition \ref{isoperi} turned out to be optimal. These results will be discussed in a forthcoming paper.



\section{Application to self-similar sets}

We use the above results to study the asymptotic behaviour as $r\to 0$ of the surface area $S_F(r)$ of the parallel sets of self-similar sets $F$ satisfying the open set condition.
In particular, we show the existence of $\cal{S}^D(F)$, provided the set $F$ is non-arithmetic, and the existence of the corresponding average limit $\widetilde{S}^D(F)$ in general. Here $D$ denotes the similarity dimension of $F$.

We start with two auxiliary results which apply to arbitrary compact sets $A\subseteq\R^d$.
Recalling the close relation between $S_A(r)$ and $(V_A)'(r)$, we 
have the following estimate.
\begin{Lemma} \label{L2}
Let $A$ be a compact subset of $\rd$ and $r_0>0$. Then for all $r>r_0$,
\[
(V_A)'_-(r)\le \left(\frac r{r_0}\right)^{d-1}(V_A)'_-(r_0).
\] 
\end{Lemma}
\begin{proof}
Let $r>r_0$. For each $0<t<r-r_0$,
\[
V_A(r)-V_A(r-t)=\int_{r-t}^r (V_A)'_+(s) ds \le t \sup_{r-t\le s\le r}(V_A)'_+(s).
\]
Since, by Stach\'o \cite[Theorem 1]{Stacho}, $s^{1-d} (V_A)'_+(s)$ is decreasing, we infer that
$(V_A)'_+(s)\le (s/r_0)^{d-1}(V_A)'_+(r_0)$
for all $r-t\le s\le r$.
Hence 
\[
\frac{V_A(r)-V_A(r-t)}t
\le \sup_{r-t\le s\le r} \left(\frac s{r_0}\right)^{d-1}(V_A)'_+(r_0)
=  \left(\frac r{r_0}\right)^{d-1}(V_A)'_+(r_0)
\]
and for $t\to 0$ we obtain
\[ (V_A)'_-(r) \le \left(\frac r{r_0}\right)^{d-1}(V_A)'_+(r_0) \le \left(\frac r{r_0}\right)^{d-1}(V_A)'_-(r_0)
\]
as claimed.
\end{proof}
Applying Corollary~\ref{Cor1} we obtain
\begin{Corollary} \label{Cor4}
Let $A$ be a compact subset of $\rd$ and $0<a<b$. Then there is a constant $c>0$ such that for all $r\in[a,b]$
\[
S_A(r)\le c.
\]
\end{Corollary}
\begin{proof}
By Corollary~\ref{Cor1}, $S_A(r)={\cal H}^{d-1}(\partial A_r)\le  {\cal H}^{d-1}(\partial A_{<r})\le {\cal M}^{d-1}(\partial A_{<r})\le (V_A)'_-(r)$
for all $r>0$. 
Hence, by Lemma~\ref{L2}, we get for all $r\in[a,b]$,
\[S_A(r)\le (V_A)'_-(r) \le \left(\frac ra \right)^{d-1} (V_A)'_-(a)\le \left(\frac ba \right)^{d-1} (V_A)'_-(a)=:c.
\]
\end{proof}

Let $F\subset\rd$ be a self-similar set generated by a function system $\{S_1,\ldots, S_N\}$ of contracting similarities $S_i:\rd\to\rd$ with contraction ratios $0<r_i<1$, $i=1,\ldots, N$. That is, $F$ is the unique nonempty, compact set invariant under the set mapping ${\bf S}(A)=\bigcup_i S_i(A)$, $A\in\rd$.  The set $F$ (or, more precisely, the system $\{S_1,\ldots,S_N\}$) is said to satisfy the \emph{open set
condition} (OSC) if there exists a non-empty, open and bounded
subset $O\subset\R^d$ 
 such that $\bigcup_i S_i O \subseteq O$ and $S_i O \cap
S_j O=\emptyset$ for all $i\neq j$. 
$F$ (or $\{S_1,\ldots,S_N\}$) is said to satisfy the \emph{strong open set condition} (SOSC), if there exist a set $O$ as in the OSC which additionally satisfies $O\cap F\neq\emptyset$. 
It is well known that OSC and SOSC are equivalent, cf.~\cite{schief}, i.e.\ for $F$ satisfying OSC, the open set $O$ can always be chosen such that $O\cap F\neq\emptyset$.

Let $D$ be the \emph{similarity dimension} of $F$, i.e.\ the unique solution $s=D$ of the equation $\sum_{i=1}^N r_i^s=1$. For $F$ satisfying OSC, $D$ coincides with the Minkowski dimension of $F$, $\dim_M F=D$. 
Finally, recall that a self-similar set $F$ is called \emph{arithmetic} (or \emph{lattice}), if there exists some number $h>0$ such
that $-\ln r_i \in h\mathbb{Z}$ for $i=1,\ldots, N$, i.e.\ if $\{-\ln r_1,\ldots,-\ln r_N\}$ generates a discrete subgroup of $\R$.
Otherwise $F$ is called \emph{non-arithmetic} (or \emph{non-lattice}). 

From the results of the previous section we immediatly derive
\begin{Proposition} \label{T3}
Let $F$ be a self-similar set satisfying OSC with similarity dimension $D<d$. Then $\overline{\dim}_S F=D$. Moreover, $\overline{\cal S}^D(F)<\infty$, i.e.\ $r^{D-d+1} S_F(r)$ is uniformly bounded as $r\to 0$.  
\end{Proposition} 
\begin{proof}
The equation $\overline{\dim}_S F=D$ follows from Corollary~\ref{cor:equpper} and the well known fact that $D=\dim_M F$. The finiteness of $\overline{\cal S}^D(F)$ is a consequence of Corollary~\ref{cor:equpper} and the finiteness of the upper Minkowski content $\overline{\cal M}^D(F)$, which is well known for self-similar sets. 
\end{proof}

We will establish below, that for non-arithmetic sets $F$ even $\dim_S F=D$ holds.  
Now we consider the S-content ${\cal S}^D(F)$ of $F$. It turns out, that in general this limit does not exist. As for the Minkowski content, Cesaro averaging improves the convergence. For a compact set $A\subset\R^d$ and $0\le s<d$, we define the \emph{average $s$-dimensional S-content} 
$\widetilde{\cal S}^s(A)$ by
\begin{equation}\label{avS-content}
\widetilde{\cal S}^s(A)=\lim_{t\to 0}\frac{1}{|\log t|}\int_t^1 \frac{S_A(r)}{(d-s)\kappa_{d-s} r^{d-1-s}}d\log r 
\end{equation}
provided this limit exists, and we write  $\overline{\widetilde{\cal S}}^s(A)$ and $\underline{\widetilde{\cal S}}^s(A)$ for the corresponding upper and lower average limits. 

\begin{Theorem} \label{T4}
Let $F\subset\R^d$ be a self-similar set satisfying OSC and 
let $D<d$ be its similarity dimension.  
Then $\widetilde{\cal S}^D(F)$ of $F$  exists and coincides with the finite and strictly positive value 
\begin{equation}\label{Xeqn}
\frac{1}{\eta} \int_0^1 r^{D-d} R(r)\ d r,
\end{equation} 
where $\eta= - \sum_{i=1}^N r_i^D \ln r_i$ and the function $R:(0,1]\to \R$ is given by
\begin{equation} \label{R-def}
R(r) = {\cal H}^{d-1}(\partial F_r)-\sum_{i=1}^N \ind{(0,r_i]}(r) {\cal H}^{d-1}(\partial(S_i F)_r)\, .
\end{equation}
If $F$ is non-arithmetic, then  also ${\cal S}^D(F)$ of $F$ exists and equals the integral in \eqref{Xeqn}.
\end{Theorem}


The proof of Theorem~\ref{T4} is postponed to the end of this section. 
We first discuss the relation of ${\cal S}^D(F)$ and the Minkowski content. If ${\cal S}^D(F)$ exists, i.e.\ if $F$ is non-arithmetic,
then both limits coincide.

\begin{Theorem}\label{T5}
Let $F$ be a non-arithmetic self-similar set satisfying OSC and let $D<d$ be its similarity dimension.  
Then ${\cal S}^D(F)= {\cal M}^D(F)$.  It follows, that ${\cal S}^D(F)>0$ and $\dim_S F = D$.
\end{Theorem}
\begin{proof} Theorem~\ref{T4} states that ${\cal S}^D(F)$ exists if $F$ is non-arithmetic and $D<d$. Hence, the equality of the contents follows from Corollary~\ref{cor4}. The equality of the dimensions is a consequence of the fact that ${\cal M}^D(F)>0$ (see~\cite[Thm. 2.4]{gatzouras}).
\end{proof}

In the arithmetic case, an analogous result holds for the average contents. We derive it from the following lemma. Recall the definitions of  $\overline{\widetilde{\cal S}}^s(A)$ and $\underline{\widetilde{\cal S}}^s(A)$ from \eqref{avS-content}. Analogously, the \emph{average $s$-dimensional Minkowski content} is given by
\begin{equation}\label{avM-content}
\widetilde{\cal M}^s(A)=\lim_{t\to 0}\frac{1}{|\log t|}\int_t^1 \frac{V_A(r)}{\kappa_{d-s} r^{d-s}}d\log r 
\end{equation}
and the corresponding $\limsup$ and $\liminf$ are denoted by
$\overline{\widetilde{\cal M}}^s(A)$ and $\underline{\widetilde{\cal M}}^s(A)$. 
 
\begin{Lemma} \label{lem:av-content}
Let $A\subset\R^d$ be compact and $0\le s<d$. Then
\begin{enumerate}
\item[(i)] $\overline{\widetilde{\cal M}}^s(A)\ge \overline{\widetilde{\cal S}}^s(A)$ and $\underline{\widetilde{\cal M}}^s(A)\ge \underline{\widetilde{\cal S}}^s(A)$
\item[(ii)] If $\overline{\cal M}^s(A)<\infty$, then $\overline{\widetilde{\cal M}}^s(A)= \overline{\widetilde{\cal S}}^s(A)$ and $\underline{\widetilde{\cal M}}^s(A)= \underline{\widetilde{\cal S}}^s(A)$.
\end{enumerate}
\end{Lemma}
\begin{proof}
For $0<t\le 1$, let 
\[
v^s(t):=\int_t^1 \frac{V_A(r)}{\kappa_{d-s}r^{d-s}} \frac{dr}{r} \quad\text{ and }\quad w^s(t):=\int_t^1 \frac{S_A(r)}{(d-s)\kappa_{d-s}r^{d-s-1}} \frac{dr}{r}.
\]
We show that
\begin{equation} \label{eq:v-w}
v^s(t)= w^s(t) +\frac{1}{d-s} \frac{V_A(t)}{\kappa_{d-s} t^{d-s}} - \frac{1}{(d-s)\kappa_{d-s}} V_A(1).
\end{equation}
By Corollary~\ref{Cor2}, we have
\[
v^s(t)=\int_t^1 \int_0^r S_A(\rho) d\rho \frac{dr}{\kappa_{d-s}r^{d-s+1}}. 
\]
Interchanging the order of integration, we get
\begin{eqnarray*}
v^s(t)&=&\frac 1{\kappa_{d-s}}\left[
\int_0^t S_A(\rho)  \int_t^1  \frac{dr}{r^{d-s+1}} d\rho + \int_t^1 S_A(\rho) \int_\rho^1 \frac{dr}{r^{d-s+1}} d\rho
\right]\\
&=& \frac 1{(d-s)\kappa_{d-s}}\left[ V_A(t) \left(\frac 1{t^{d-s}}-1\right) + \int_t^1 S_A(\rho) \left(\frac 1{\rho^{d-s}}-1\right) d\rho\right]\\
&=& \frac 1{d-s} \frac{V_A(t)}{\kappa_{d-s} t^{d-s}} +  w^s(t) +  \frac {1}{(d-s)\kappa_{d-s}} \left( -V_A(t)-V_A(1)+V_A(t)\right),
\end{eqnarray*} 
where we used again the relation $V_A(r)=\int_0^r S_A(\rho) d\rho$. This proves \eqref{eq:v-w}. 

Observe that the third term on the right in \eqref{eq:v-w} is constant. It vanishes, when dividing by $|\log t|$ and taking the limit as $t\to\infty$. The second term is non-negative.   
Let $(t_n)$ be a null sequence, such that 
\[
\lim_{n\to\infty} \frac{w^s(t_n)}{|\log t_n|}=\overline{\widetilde{\cal S}}^s(A).
\]
Then
\[
\overline{\widetilde{\cal M}}^s(A)\ge \limsup_{n\to\infty} \frac{v^s(t_n)}{|\log t_n|}\ge \overline{\widetilde{\cal S}}^s(A).
\]
Similarly the inequality $\underline{\widetilde{\cal M}}^s(A)\ge \underline{\widetilde{\cal S}}^s(A)$ is obtained by choosing a sequence $(t_n)$ such that $\underline{\widetilde{\cal M}}^s(A)$ is attained.

If $\overline{\cal M}^s(A)<\infty$ holds, then the second term on the right in \eqref{eq:v-w} is bounded by a constant. Hence, it vanishes when dividing by $|\log t|$ and taking the limit as $t\to\infty$. The stated equalities follow at once.
\end{proof}

\begin{Theorem}\label{Thm:avlim}
Let $F$ be a self-similar set satisfying OSC and let $D<d$ be the similarity dimension of $F$.  
Then $\widetilde{\cal S}^D(F)= \widetilde{\cal M}^D(F)$.
\end{Theorem}
\begin{proof}
By Theorem~\ref{T4}, the average S-content of $F$ exists, i.e.~${\overline{\widetilde{\cal S}}}^D(F)=\underline{\widetilde{\cal S}}^D(F)$. Since $\overline{\cal M}^D(F)<\infty$ (as is well known and easily verified), Lemma~\ref{lem:av-content} (ii) implies the assertion. 
\end{proof}


The proof of Theorem~\ref{T4} is based on the following estimates. 
Fix a feasible set $O$ satisfying the SOSC, i.e.\ with $O\cap F\neq\emptyset$. Let $C:=\bigcup_{i=1}^N S_i O$. 
The following lemma gives an upper bound for the growth of the surface area of $F_r$ near the boundary of $C$ as $r\to 0$. 
\begin{Lemma}\label{keylem}
There exist constants $c,\gamma>0$ such that for all $0<r\le 1$ 
\[{\cal H}^{d-1}(\partial F_r \cap (\R^d\setminus C)_r)\le c r^{d-1-D+\gamma}. \]
\end{Lemma}
\begin{proof} 
Let $ \Sigma^*:=\bigcup_{n=0}^\infty \{1,\ldots, N\}^n$ and, for $0<t\le 1$, let 
\begin{equation}\label{eq:Sigma_rho}
\Sigma(t)=\{w=w_1\ldots w_n\in\Sigma^*: r_w< t \le r_w r_{w_n}^{-1}\},
\end{equation}	
where $r_w:=r_{w_1}\ldots r_{w_n}$. Similarly, we will use $S_w:=S_{w_1}\circ \ldots \circ S_{w_n}$.	For convenience, let $\Sigma(t)$ for $t>1$ be the set just containing the empty sequence $\tau$ and set $r_{\tau}:=1$ and $S_\tau:={\rm id}$. 
Furthermore, let $r_{\min}:=\min_{1\le i\le N} r_i$. 

For a closed set $B\subseteq\R^d$ and 
$\eps>0$, let 
\begin{equation}\label{SigmaBdef}
\Sigma(B,\eps)=\{w\in\Sigma(\rho^{-1}\eps): (S_w F)_\eps\cap B \neq\emptyset\},
\end{equation}
where $\rho$ is a constant we will fix later.
First we show that there is a constant $c'>0$ (independent of $B$) such that for all $\eps>0$
\begin{equation}\label{eq:sigB}
{\cal H}^{d-1}(\partial F_\eps\cap B)\le c \,\#{\Sigma(B,\eps)}\, \eps^{d-1}. 
\end{equation}
For $\eps>0$, the relation $F_\eps\cap B = \bigcup_{w\in\Sigma(B,\eps)}(S_wF)_\eps\cap B$ implies that
\begin{eqnarray*} 
{\cal H}^{d-1}(\partial F_\eps \cap B)&\le& {\cal H}^{d-1}(\bigcup_{w\in\Sigma(B,\eps)} \partial (S_wF)_\eps )\\
&\le& \sum_{w\in\Sigma(B,\eps)} {\cal H}^{d-1}(\partial(S_wF)_\eps )  \\
&\le& \sum_{w\in\Sigma(B,\eps)} r_w^{d-1} {\cal H}^{d-1}(\partial F_{\eps/r_w}).
\end{eqnarray*}
By definition of $\Sigma(\rho^{-1}\eps)$, $a:=\rho<\eps/r_w\le \rho r_{\min}^{-1}=:b$. Hence, by Corollary~\ref{Cor4}, ${\cal H}^{d-1}(\partial F_{\eps/r_w})$ is bounded by some constant $c>0$ uniformly for all $\eps>0$ and $w\in\Sigma(\rho^{-1}\eps)$.
Since $r_w\le \rho^{-1} \eps$, we obtain
\begin{equation*} 
{\cal H}^{d-1}(\partial F_\eps \cap B)\le c \sum_{w\in\Sigma(B,\eps)} (\rho^{-1} \eps)^{d-1} = c' \#{\Sigma(B,\eps)} \eps^{d-1} ,
\end{equation*}
with $c':=\rho^{1-d} c$. This completes the proof of (\ref{eq:sigB}).

Now set $B:=(\R^d\setminus C)_r$. To derive an upper bound for the cardinality of $\Sigma((\R^d\setminus C)_r,\eps)$, we apply \cite[Lemma~5.4.1]{winter} with the choice $r=1$ and $\varepsilon=\delta$. Note that the set $O(1)$ in \cite{winter} equals $C$. The lemma requires to choose $\rho$ as in \cite[(5.1.8), also cf.~the paragraph preceding it]{winter}. 
We infer, that there are constants $\tilde c,\gamma>0$ such that 
\begin{equation} \label{eq:sigB2}  
{\#}\Sigma((\R^d\setminus C)_r ,\eps)\le \tilde c \eps^{\gamma-D} 
\end{equation}
for all $0<\eps\le\rho$. By adjusting the constant $\tilde c$, the estimate can be adapted to hold for all $r\in(0,1]$,  since for $\eps\ge\rho$, the cardinality of $\Sigma(B,\eps)$ is bounded by $\#\Sigma(B,\rho)$. 
Now the assertion follows by combining (\ref{eq:sigB}) and (\ref{eq:sigB2}).
\end{proof}

Applying Lemma~\ref{keylem}, we derive the following estimate for the function $R$ in \eqref{R-def}.

\begin{Lemma} \label{R_est}
There exist $c,\gamma > 0$ such that for all $0<\eps\le 1$ 
\begin{equation*}
|R(\eps)|\le c \eps^{d-1-s+\gamma}.
\end{equation*}
\end{Lemma}
\begin{proof}
We abbreviate ${\cal H}:={\cal H}^{d-1}$. Fix $0<\eps<r_{\min}$. Set $U:=\bigcup_{i\neq j} (S_i F)_\eps\cap (S_j F)_\eps$ and $B^j:=(S_j F)_\eps\setminus U$. 
Then $F_\eps=\bigcup_j B^j\cup U$ is a disjoint union and so 
\[{\cal H} (\partial F_\eps)=\sum_{j=1}^N {\cal H} (\partial F_\eps\cap B^j)+{\cal H} (\partial F_\eps \cap U).\]
Similarly,
\[{\cal H} (\partial(S_j F)_\eps)={\cal H} (\partial(S_j F)_\eps \cap B^j)+{\cal H} (\partial(S_j F)_\eps\cap U),\]
since $(S_j F)_\eps\subseteq B^j\cup U$.
Hence $R(\eps)$ can be written as
\[R(\eps)= \sum_{j=1}^N \left({\cal H} (\partial F_\eps \cap B^j)-{\cal H} (\partial(S_jF)_\eps\cap B^j)\right) + {\cal H} (\partial F_\eps\cap U) - \sum_{j=1}^N {\cal H} (\partial (S_j F)_\eps\cap U).\]
Observe that $\partial F_\eps\cap B^j=\partial(S_jF)_\eps\cap B^j$. 
Therefore, all terms of the first sum on the right are zero. Taking absolute values, we infer 
\begin{equation}\label{Rabsest1}
|R(\eps)| \le {\cal H} (\partial F_\eps\cap U) + \sum_{j=1}^N {\cal H} (\partial(S_j F)_\eps\cap U).
\end{equation}
For the first term note that
$U\subseteq (\R^d\setminus C)_\eps$ (Fact I; see proof below). Recall that $C=\bigcup_{i=1}^N S_i O$.  
For the remaining terms in (\ref{Rabsest1})
we have 
\[ 
{\cal H}(\partial(S_jF)_\eps\cap U)= 
r_j^k {\cal H} (\partial F_{\eps/r_j}\cap S_j^{-1}U)
\le r_j^k {\cal H} (\partial F_{\eps/r_j}\cap (\R^d\setminus C)_{r_j/r}),
\]
where the inequality is due to the set inclusion $F_{\eps/r_j}\cap S_j^{-1}U \subseteq (O\setminus\R^d)_{\eps/r_j}\subseteq (\R^d\setminus C)_{\eps/r_j}$ (Fact II; see proof below).
We obtain for each $0<\eps\le r_{\min}$, 
\begin{equation}\label{Rabsest2}
|R(\eps)| \le {\cal H} (\partial F_\eps\cap (\R^d\setminus C)_\eps) + \sum_{j=1}^N r_j^{d-1} {\cal H} (\partial F_{\eps/r_j} \cap (\R^d\setminus C)_{\eps/r_j}).
\end{equation}
By Lemma~\ref{keylem}, for each of the terms in (\ref{Rabsest2}) there are constants $c,\gamma>0$ such that the term is bounded from above by $c \eps^{d-1-D+\gamma}$ for $0<\eps\le r_{\min}$. Hence we can also find such constants for $|R(\eps)|$. The estimate can be adapted to hold for all $0<r\le 1$ by suitably enlarging $c$, since, by Corollary~\ref{Cor4}, each of the terms of $R(r)$ in \eqref{R-def} is bounded by a constant for all $\eps\in [r_{\min},1]$. 
It remains to verify the two set inclusions used above.

\emph{Proof of Fact I ($U\subseteq (\R^d\setminus C)_\eps$)}: 
Let $x\in U$. We show that $d(x, \R^d\setminus C)\le\eps$ and thus $x\in(\R^d\setminus C)_\eps$.
Assume $d(x, \R^d\setminus C)>\eps$. Since the union $C=\bigcup_i S_i O$ is disjoint, there is  
a unique $j$ such that $x\in S_jO$. Moreover, $d(x,\partial S_jO)>\eps$. Since $x\in U$, there is at least one index $i\neq j$ such that $x\in(S_iF)_\eps$ and consequently a point $y\in S_iF$ with $d(x,y)\le\eps$. But then $y\in S_i F\cap S_j O$, a contradiction to OSC. Hence, $d(x, \R^d\setminus C)\le\eps\,$.

\emph{Proof of Fact II ($F_{\eps/r_j}\cap S_j^{-1}U\subseteq (O\setminus\R^d)_{\eps/r_j}$)}: Let $x\in F_{\eps/r_j}\cap S_j^{-1}U$. Then $S_jx\in U$ and so there exists at least one index $i\neq j$ with $S_jx\in (S_iF)_\eps$. Hence $d(S_jx,\partial S_jO)\le\eps$ since otherwise there would exist a point $y\in S_iF\cap S_jO$, a contradiction to OSC.
Therefore, $d(x,\partial O)\le \eps/r_j$, i.e.\ $x\in(O\setminus\R^d)_{\eps/r_j}$. 
\end{proof}

To complete the proof of Theorem~\ref{T4}, we apply the following slight improvement of Theorem~4.1.4 in \cite{winter}, in which we replace the assumption of continuity 
off a discrete set by continuity almost everywhere.

\begin{Theorem} \label{adrenewalthm}
Let $F$ be a self-similar set with ratios $r_1,\ldots,r_N$ and similarity dimension $D$. 
For a function $f:(0,\infty)\to\R$, suppose that for some $k\in\R$ the function $\varphi_k$ defined by 
\begin{equation*}
\varphi_k(\eps)=f(\eps)-\sum_{i=1}^N r_i^{k} \textbf{\upshape 1}_{(0,r_i]}(\eps)  f(\eps/r_i)
\end{equation*} 
is continuous at Lebesgue almost every $\eps>0$ and satisfies
\begin{equation}\label{adapteqn1}
 |\varphi_k(\eps)|\le c \eps^{k-D+\gamma}
\end{equation}
for some constants $c,\gamma>0$ and all $\eps>0$. 
Then $\eps^{D-k} f(\eps)$ is uniformly bounded in $(0,\infty)$ and the following holds:
\begin{enumerate}
\item[(i)] The limit $\underset{\delta\to 0}{\lim}\frac{1}{|\ln \delta|}\int_{\delta}^1 \eps^{D-k} f(\eps)\frac{d\eps}{\eps}$ exists and equals
\begin{equation}\label{adapteqn2}
\frac{1}{\eta} \int_0^1 \varepsilon^{D-k-1}
\varphi_k(\varepsilon)\ d\varepsilon \, ,
\end{equation}
where  $\eta= - \sum_{i=1}^N r_i^D \ln r_i$.
\item[(ii)] If $F$ 
 is non-arithmetic, then the limit of $\eps^{D-k} f(\eps)$ as $\eps\to 0$ exists and equals the expression in \eqref{adapteqn2}.
\end{enumerate}
\end{Theorem}
 
\begin{proof}
The arguments in the proof of Theorem~4.1.4 in \cite{winter} can easily be adapted to derive the statement from the Renewal Theorem in Feller~\cite[p.363]{Feller}.
The assumptions on $\varphi_k$ ensure that the function $z:\R\to\R$ defined by
$z(t)= e^{(k-D)t}\varphi_k(e^{-t})$ for $t\ge 0$ and $z(t)=0$ for $t<0$
is directly Riemann integrable, see \cite[p.362]{Feller} or \cite[p.118]{Asmussen} for a definition. (If $z$ is bounded and continuous Lebesgue a.e.\ and bounded from above and below by some directly Riemann integrable functions, then $z$ is directly Riemann integrable, cf.\ for instance \cite[Prop.~4.1, p.118]{Asmussen}. Clearly, $e^{-\gamma t}$ is directly Riemann integrable.)
\end{proof}

\begin{proof}[Proof of Theorem~\ref{T4}] Apply Theorem~\ref{adrenewalthm} with $f(r):={\cal H}^{d-1}(\partial F_r)$, $k:=d-1$ and $\varphi_k(r):=R(r)$. Lemma~\ref{R_est} ensures that the hypothesis \eqref{adapteqn1} is satisfied. The continuity of $R$ a.e.\ follows from the same property of $\Ha^{d-1}(\partial A_r)$ for sets $A\subseteq \R^d$ (cf.\ Corollary~\ref{Cor2} and \cite[Lemma~2, p.367]{Stacho}).
\end{proof}

\section{Applications to random sets}

A {\it random compact set} in $\rd$ is a measurable mapping 
$$Z: (\Omega,{\cal A},\Pr)\to ({\cal K}',{\cal B}({\cal K}')),$$
where ${\cal K}'$ is the family of all nonempty compact subsets of $\rd$ and ${\cal B}({\cal K}')$ is the Borel $\sigma$-algebra on ${\cal K}'$ equipped with the Hausdorff distance (cf.\ \cite{M75}).
 
\begin{Theorem} \label{TR}
 Let $Z\subseteq\rd$ be a random compact set. If the function $r\mapsto \E V_Z(r)$ is differentiable at some point $r>0$, then 
$V_Z'(r)={\cal H}^{d-1}(\partial Z_r)$ almost surely and $(\E V_Z)'(r)=\E{\cal H}^{d-1}(\partial Z_r)$.
\end{Theorem}
 
\begin{proof}
First we show that
\begin{equation}  \label{E0}
\E (V_Z)'_-(r)<\infty\quad\mbox{for all } r>0.
\end{equation}
Indeed, if \eqref{E0} would not hold for some $r>0$ then, due to Lemma~\ref{L2}, we would have $\E (V_Z)'_-(s)=\infty$ for all $s<r$, which would imply (using Tonelli's theorem)
$$\E V_Z(r_0)=\int_0^{r_0}\E(V_Z)'_-(s)\, ds=\infty,\quad r_0<r,$$
contradicting the assumptions. 

Next, let $r>0$ be such that $(\E V_Z)'(r)$ exists. We show that
\begin{equation} \label{E1}
\E (V_Z)'_-(r)=(\E V_Z)'_-(r),\quad \E (V_Z)'_+(r)=(\E V_Z)'_+(r).
\end{equation} 
Choose any $0<r_0<r$.
Since $V_Z(r)-V_Z(r-t)=\int_{r-t}^r (V_Z)'_-(s)\, ds$ for any $t<r-r_0$, we have 
$$\frac{V_Z(r)-V_Z(r-t)}{t}\leq\sup_{r-t<s<r}(V_Z)'_-(s)\leq \left( \frac r{r_0}\right)^{d-1}(V_Z)'_-(r_0)$$
by Lemma~\ref{L2}.
The last random variable is integrable by \eqref{E0}, hence, applying the Lebesgue Dominated Theorem for $t\to 0_+$, the first equality in \eqref{E1} is verified. The second identity follows analogously.
Consequently, we have
$$0=(\E V_Z)'_-(r)-(\E V_Z)'_+(r)=\E\left((V_Z)'_-(r)-(V_Z)'_+(r) \right),$$
and, hence, the left and right hand side derivatives are equal and $(V_Z)'(r)={\cal H}^{d-1}(\partial Z_r)$ almost surely. Using \eqref{E1} again, we can apply expectation on both sides and interchange derivative with expectation on the left hand side, which completes the proof.
\end{proof}

Let now $Z=S_t$ be the trajectory of a Brownian motion up to given time $t>0$ and let $S_{r,t}$ be its $r$-parallel set (called also Wiener sausage). It is shown in \cite[Theorem~2.2]{RSS07} and \cite[Theorem~4.5]{Last06} that for all $t>0$, $\E\Ha^{d-1}(S_{r,t})=\frac{\partial}{\partial r}\E \Ha^d(\partial S_{r,t})$ for almost all $r>0$, where the words ``almost all'' can be dropped in dimension $d\leq 3$. (An exact formula for the mean volume of the Wiener sausage is known, see \cite{RSS07} and the references therein.) As a corollary of Theorem~\ref{TR}, we can show that the above mentioned results are true for all $r>0$ in any dimension.

\begin{Corollary}
Let $S_{r,t}$ be the parallel $r$-neighbourhood of the trajectory of a standard Brownian motion in $\rd$ on the time interval $[0,t]$, where $r,t>0$ are arbitrary fixed. Then
\begin{eqnarray*}
\lefteqn{\E \Ha^{d-1}(\partial S_{r,t})=\frac{\partial}{\partial r}\E \Ha^d(S_{r,t})}\\
&=&d\kappa_dr^{d-1}\left( 1+(d-2)^2\frac t{2r^2}+\frac{4d}{\pi^2}
\int_0^{\infty}\frac{\varphi_d(y^2\frac{t}{2r^2})}{y^3(J_{\frac{d-2}2}^2(y)+Y_{\frac{d-2}2}^2(y))}\, dy\right) ,
\end{eqnarray*}
where $\varphi_d(z)=1-e^{-z}-2ze^{-z}/d$ and $J_\nu$, $Y_\nu$ are the Bessel functions of first and second type, respectively, and order $\nu$. 
\end{Corollary}

In the sequel, we shall consider the asymptotic behaviour of the volume and surface area of a Brownian motion $Z=S_1$ on the time interval $[0,1]$. (Note that related results for $S_t$ with a general $t>0$ can be easily derived, using the stochastic self-similarity of the Brownian motion.) The asymptotic behaviour of the volume of a parallel set is well known, both almost surely and in the mean. We have, both almost surely and in the mean (see \cite{LeGall88})
\begin{eqnarray}
\lim_{r\to 0} |\log r|\Ha^2(Z_r)=\pi,\quad &&d=2,\label{q1}\\
\lim_{r\to 0}\frac{\Ha^d(Z_r)}{\kappa_{d-2}r^{d-2}} =(d-2)\pi ,\quad &&d\geq 3.\label{q2}
\end{eqnarray}
From the known integral representation of the mean volume, it has been derived in \cite{RSS07} that
\begin{eqnarray}
\lim_{r\to 0}r|\log r|^2\E\Ha^1(\partial Z_{r})=\pi,\quad &&d=2,\label{q3}\\
\lim_{r\to 0}\frac{\E\Ha^{d-1}(\partial Z_{r})}{(d-2)\kappa_{d-2} r^{d-3}}=(d-2)\pi ,\quad &&d\geq 3.\label{q4}
\end{eqnarray}

Equations \eqref{q1} and \eqref{q2} imply immediately that $\dim_MZ=2$ almost surely (for $d\ge 2$).
Unfortunately, S-content and S-dimension of $Z$ cannot be derived so easily. Using the methods from Section~3, we get the following estimates.

\begin{Proposition}
If $d=2$ we have almost surely
\begin{eqnarray}
\label{p1} \limsup_{r\to 0}r|\log r|\Ha^1(\partial Z_r)&\leq& 2\pi,\\
\label{p2} \liminf_{r\to 0}\sqrt{|\log r|}\Ha^1(\partial Z_r)&\geq& 2\pi.
\end{eqnarray}
Hence, $1\leq\underline{\dim}_SZ\leq \overline{\dim}_SZ =2$ almost surely.

For $d\geq 3$, we have almost surely
\begin{eqnarray}
\label{p3} \limsup_{r\to 0}\frac{\Ha^{d-1}(\partial Z_r)}{(d-2)\kappa_{d-2}r^{d-3}}&\leq& \frac{(d-2)^2}d\pi,\\
\label{p4} \liminf_{r\to 0}\frac{\Ha^{d-1}(\partial Z_r)}{(d-2)\kappa_{d-2}r^{d-3-2/d}}&>&0 .
\end{eqnarray}
Hence, $2-\frac 2d\leq\underline{\dim}_SZ\leq\overline{\dim}_SZ=2$ almost surely.
\end{Proposition}

\begin{proof}
In order to obtain \eqref{p1}, we use a similar method as in the proof of Lemma~\ref{L-34}. Assume, to the contrary, that $r_i\to 0$ is such that
$$\lim_i r_i|\log r_i|\Ha^1(\partial Z_{r_i})> 2\pi+2\ep$$
for some $\ep>0$. Then we have for $i$ sufficiently large, 
\begin{eqnarray*}
\Ha^2(Z_{r_i})&\geq&\sum_{j=i}^\infty \Ha^1(\partial Z_{r_j}) \frac{r_j^2-r_{j+1}^2}{2r_j}\\
&=&\sum_{j=i}^\infty r_j|\log r_j|\Ha^1(\partial Z_{r_j}) \frac{r_j-r_{j+1}}{r_j|\log r_j|}\frac{r_j+r_{j+1}}{2r_j}\\
&\geq& (\pi +\ep)\sum_{j=i}^\infty \frac{r_j-r_{j+1}}{r_j|\log r_j|}\\
&\geq& (\pi +\ep)\sum_{j=i}^\infty \left(\frac 1{|\log r_j|}-\frac 1{|\log r_{j+1}|}\right)\\
&=&(\pi +\ep)\frac 1{|\log r_i|}
\end{eqnarray*}
(the last inequality follows from the concavity of the function $t\mapsto |\log t|^{-1}$ on $[0,e^{-1/2}])$. Hence, $\Ha^2(Z_{r_i})|\log r_i|\geq \pi+\ep$ for sufficiently large $i$, which is a contradiction.

The lower bound \eqref{p2} follows by the isoperimetric inequality. Indeed, we have 
$$4\pi\Ha^2(Z_r)\leq \Ha^1(\partial Z_r)^2,$$
hence,
$$2\sqrt{\pi}\sqrt{|\log r|\Ha^2(Z_r)}
\leq\sqrt{|\log r|}\Ha^1(\partial Z_r),$$
and \eqref{p2} follows using \eqref{q1}.

\eqref{p3} follows from Lemma~\ref{L-34} and \eqref{p4} can be obtained by using the isoperimetric inequality again, as in the proof of Proposition~\ref{isoperi}.
\end{proof}

\end{document}